\documentclass[12pt]{amsart}
\usepackage{amsmath}

\def\lf{\left}
\def\ri{\right}
\def\a{{\alpha}}

\def\wt{\widetilde}

\def\K{K\"ahler }

\def\A{Amp\`{e}re }

\def\be{\begin{equation}}
\def\ee{\end{equation}}

\def\lf{\left}
\def\ri{\right}
\def\a{{\alpha}}

\def\wt{\widetilde}

\def\wt{\widetilde}

\long\def\symbolfootnote[#1]#2{\begingroup%
\def\thefootnote{\fnsymbol{footnote}}\footnote[#1]{#2}\endgroup}

\newtheorem{thm}{Theorem}

\newtheorem{cor}{Corollary}
\theoremstyle{definition}

\theoremstyle{remark}
\newtheorem{rem}{Remark}
\begin{document}
\title[ A note on harmonic forms and the K\"ahler cone]
{ A note on harmonic forms and the boundary of the K\"ahler cone}

\author{Albert Chau$^1$}

\address{Department of Mathematics,
The University of British Columbia, Room 121, 1984 Mathematics
Road, Vancouver, B.C., Canada V6T 1Z2} \email{chau@math.ubc.ca}

\author{Luen-Fai Tam$^2$}
\address{The Institute of Mathematical Sciences and Department of
 Mathematics, The Chinese University of Hong Kong,
Shatin, Hong Kong, China.} \email{lftam@math.cuhk.edu.hk}
\thanks{$^1$Research
partially supported by NSERC grant no. \#327637-11}
\thanks{$^2$Research partially supported by Hong Kong RGC General Research Fund
\#CUHK 403108}

\begin{abstract}
Motivated by the results of Wu-Yau-Zheng \cite{WuYauZheng}, we show that under a certain curvature assumption the harmonic representative of any boundary class of the \K cone is nonnegative.
\end{abstract}

\maketitle

 The purpose of this note is to prove the following:
\begin{thm}\label{t1}
Let $(M^n, g)$ be a compact complex n-dimensional \K manifold satisfying the following curvature condition: for any $x\in M$, unitary frame $\{e_1,\dots,e_n\}$ of $T^{(1,0)}_x(M)$ and any real numbers $\xi_1,...,\xi_n$ we have
\begin{equation}\label{e1}
\sum_{i,j=1}^n R_{i\bar{i}j\bar{j}} (\xi_i -\xi_j)^2 \geq 0.
\end{equation}
Let $\alpha$ be in the closure of the \K cone of $M$ and $\eta$ be the unique harmonic representative in $\alpha$.  Then $\eta$ is nonnegative.  Moreover, $\eta$ is positive if and only if $\alpha^n[M] >0$.
\end{thm}

Our motivation comes from the results of Wu, Yau and Zheng in \cite{WuYauZheng} where the authors studied a degenerate complex Monge \A equation to better understand the boundary of the K\"ahler cone under the curvature condition in \eqref{e1}.   We will give two proofs of Theorem \ref{t1}, and one purpose here is to point out that a rather straight forward observation made on the proof in \cite{WuYauZheng} leads to a proof of Theorem \ref{t1}.  After showing this, we then present another short self-contained  proof of Theorem \ref{t1}. \symbolfootnote[1]{After posting the first version of this note, it was pointed out to us that the nonnegativity of $\eta$ in Theorem \ref{t1} had been obtained by Zhang in \cite{Z}.}

 Before we begin, let us first recall some basic definitions and notation.  Given a complex manifold $M$,  recall that a real class $\alpha\in H^{(1,1)}(M)$ is called a \K class if $\alpha$ contains a smooth positive definite representative $\eta$.  The space of \K classes is a convex cone in $H^{(1,1)}(M)$  referred to as the \K cone which we denote by $\mathcal{K}$.  We say that $\alpha$ is in the closure of $\mathcal{K}$ if $[(1-t)\omega+t\eta] \in\mathcal{K}$ for any smooth $\eta\in\alpha$, $\omega\in\mathcal{K}$ and $t\in[0,1)$.  Finally, given any real $\alpha\in H^{(1,1)}(M)$ we use $\alpha^n[M]$ to denote the integral $\int_M  \alpha^n$.

The following is proved in  \cite{WuYauZheng}:
\begin{thm}\label{WYZ}[Wu-Yau-Zheng]
Let $(M^n, g)$ be a compact complex n-dimensional \K manifold satisfying the curvature condition in Theorem \ref{t1}.
Then any boundary class of the \K cone  can be represented by a $C^\infty$ closed $(1,1)$ form that is everywhere nonnegative.
\end{thm}

In particular, they prove that: if $\omega_0$ is the \K form of $(M,g)$ and $\Phi\in H^{(1,1)}(M)$ is real and satisfies $[\omega_0+t\Phi]\in \mathcal{K}$ for $0\le t<1$ and
$$\int_M(\omega_0+\Phi)^n=0,
$$
then there exists a smooth solution $v$ to \begin{equation}\label{eq-WYZ}
\begin{cases}
(\omega_0+\Phi+dd^cv)^n&=0,\\
\omega_0+\Phi+dd^cv&\ge0
\end{cases}
\end{equation}
 on $M$.  Note that $[\omega_0+t\Phi]\in\mathcal{K}$ for $0\le t<1$ iff $[\omega_0+ \Phi]$ is in the closure of $\mathcal{K}$.   Thus the solvability of \eqref{eq-WYZ} for any $\Phi$ above is equivalent to the statement of Theorem \ref{WYZ} for boundary classes $\alpha$ satisfying $\alpha^n[M] =0$.  Their proof is to consider smooth solutions $v_t$ to

\begin{equation}\label{eq-WYZ2}
\begin{cases}
(\omega_0+t\Phi+dd^c v_t)^n&=\gamma(t) \omega_0^n,\\
\omega_0+t\Phi+dd^c v_t&>0
\end{cases}
\end{equation}
for each $0\le t<1$ where $\gamma(t)$ is the normalizing factor
$$
\gamma(t)=\frac1{V(M,g)}\int_M(\omega_0+t\Phi)^n,
$$
and the existence of each $v_t$ is guaranteed by the results of Yau \cite{Y}.   The solution to \eqref{eq-WYZ} is then obtained by letting $t\to 1$.  In fact, under the curvature assumptions in Theorem \ref{t1} they show that $v_t=tv$ where $v$ is a fixed function independent of $t$ and hence solves \eqref{eq-WYZ} by letting $t\to 1$ in \eqref{eq-WYZ2}.  We are now ready to present:

\begin{proof}[First proof of Theorem \ref{t1}]
Let $v_t=tv$ be as above.  We begin by showing that $v$ satisfies a rather special property.  Since
$$
(\omega_0+t(\Phi+dd^c v))^n=\gamma(t)\omega_0^n
$$
 for all $0\le t<1$ by \eqref{eq-WYZ2}, if we let $a_1(x), \dots,a_n(x)$ be the ordered eigenvalues of $\Phi+dd^cv$ at $x\in M$ (with respect to $\omega_0$) we obtain
 $$
 \prod_{i=1}^n(1+ta_i(x))=\gamma(t)
 $$
 for all $0\le t<1$. Since the RHS does not depend on $x$, it is not hard to show that the $a_i(x)'s$ are constant functions on $M$ for each $i$, in other words the eigenvalues of $\Phi+dd^c v$ with respect to $\omega_0$ must be the same at each point on $M$. In particular, the trace of $\omega_0+\Phi+dd^cv$ is constant.  Suppose now that $\Phi$ had been chosen as the unique harmonic representative in $[\Phi]$.  Thus $\omega_0+\Phi$ is also harmonic and as pointed out in \cite{WuYauZheng}, it follows by the curvature assumption in \eqref{e1} that $\omega_0+\Phi$ must also be parallel.  In particular, the trace of $\omega_0+\Phi$ is constant and hence the trace of $dd^cv$ is constant as well.  Since $M$ is compact, it follows that $v$ must also be   constant on $M$.   Equivalently, we can summarize this as: in general (without assuming $\Phi$ is harmonic), \eqref{eq-WYZ} is always satisfied by any $v$ such that $\omega_0+\Phi+dd^c v$ is harmonic.  In other words, we have shown that the conclusion of Theorem \ref{WYZ} is always satisfied by the harmonic representative of any boundary class, and thus we have proved  the first  statement in Theorem \ref{t1}. The proof of the second statement of Theorem \ref{t1} is the same as  that in the Second proof of Theorem \ref{t1} below. \end{proof}
\begin{rem}
It has been known for some time that under the stronger assumption of nonnegative holomorphic bisectional curvature a harmonic $(1,1)$ form must be parallel  (see for example  \cite{GK, HowardSmythWu1981} and references therein).  This fact played a key role in the classification results for nonnegatively curved \K manifolds in \cite{HowardSmythWu1981} for example.  The proof of parallelism uses the Bochner formula for $(1,1)$ forms on \K manifolds and generalizes immediately to the curvature condition \eqref{e1}.  \end{rem}

One may ask if it can directly be proved that the harmonic representative of a boundary class above is actually nonnegative.  We present this in the following.

\begin{proof} [Second proof of Theorem \ref{t1}]  
 Let $\eta$ be as in Theorem \ref{t1} and let $\omega_0$ be the \K form for $(M, g)$.  By the above remarks, $\eta$ is parallel and thus has constant real eigenvalues $a_1,...,a_n$ on $M$ with respect to $\omega_0$.  Also, $[(1-t)\omega_0+t\eta]\in \mathcal{K}$  for every $t\in[0,1)$.

  In other words, for each $t\in[0,1)$ there exists $f_t\in C^{\infty}(M)$ and $\omega_t \in \mathcal{K}$ such that $(1-t)\omega_0+t\eta=\omega_t +dd^c f_t$, giving
$$
Vol_g (M)\prod_{i=1}^n(1-t+ta_i)= \int_M((1-t)\omega_0+t\eta)^n=\int_M\lf(\omega_{t}+dd^c f_{t}\ri)^n>0
$$
for all $t\in[0,1)$.  On the other hand, if $a_k<0$ for some $k$ then $1-t+ta_k$ and thus the product on the LHS above would vanish for some $t_0\in(0,1)$ giving a contradiction. Thus $a_i$ must be nonnegative for each $i$, in other words $\eta$ is nonnegative.  In particular, we have $\int_M \eta^n \geq  0$ with strict inequality if and only if $\eta$  is positive.  This completes the proof.
\end{proof}

The fact that the harmonic form $\eta$ is parallel allows for a  corresponding decomposition of the universal cover $\wt M$ of $M$ by the de Rham decomposition Theorem for \K manifolds. This in turn will allow a further description of the boundary of $\mathcal{K}$.  Let $\wt M$ be the universal cover of $M$ with projection $\pi:\wt M\to M$.  By the de Rham decomposition Theorem for \K manifolds, we may write $$(\wt M, \wt \omega_0)=(\wt M_0, \wt \sigma_0)\times (\wt M_1, \wt \sigma_1)\times\cdots\times (\wt M_k , \wt \sigma_k)$$ where $\wt \omega_0=\pi^*(\omega_0)$, each factor on the RHS is irreducible and \K and the decomposition is unique up to permutation.  In the following we will identify $\pi_1(M)$, the first fundamental group of $M$, with the corresponding group of deck transformations of $\wt{M}$.

 \begin{cor}
With the above notations, the boundary of $\mathcal{K}$ can be identified with the space of harmonic $(1,1)$ forms $\wt \eta$ on $\wt M$ satisfying: $\wt \eta$ is equivariant with respect to $\pi_1(M)$ and $\tilde\eta=\prod_{i=1}^k a_i\wt \sigma_i$, where $a_i\ge0$ for all $i$ with equality holding for some $i$.
\end{cor}
\begin{proof} Let $\wt \eta$ be a harmonic form on $\wt M$ as above.  Then $\wt \eta$ descends to a harmonic form $\eta$ and it is easy to see that $[\eta]$ is in the boundary of $\mathcal{K}$. Note that the map $\wt \eta \to [\eta]$ is one-one.

On the other hand, if $\a$ is in the boundary of $\mathcal{K}$ and $\eta$ is the unique harmonic representative in $\alpha$, then $\eta$ is nonnegative by Theorem \ref{t1} and also parallel. Hence the eigenvalues of $\eta$ are nonnegative constants.
Thus $\wt \eta=\pi^*(\eta)$ is likewise harmonic with nonnegative constants.   By the de Rham decomposition theorem for \K manifolds, $\wt M$ splits into a product $\wt N_0\times\wt N_1\times\cdots\times \wt N_l$ such that $\tilde\eta$ splits accordingly as $\wt \eta_0\times\wt \eta_1\times\cdots\times \wt \eta_l$ where $\wt \eta_0$ is the zero form on $\wt N_0$ and $\wt \eta_i$ is a positive multiple of the \K form on each $\wt N_i$, $1\le i\le l$. By the uniqueness of de Rham decomposition, by further decomposing $N_i$ into irreducible factors one can see that $\tilde \eta$ is a harmonic form on $\wt M$ as in the theorem.

\end{proof}
\bibliographystyle{amsplain}

\end{document}